\documentclass{amsart}

\usepackage{amssymb}
\usepackage{amsthm}
\usepackage{amsmath} 
\usepackage{amsbsy}
\usepackage{hyperref}
\usepackage{tikz}
\newtheorem{theorem}{Theorem}

\newtheorem{proposition}{Proposition}

\newtheorem*{lem}{Lemma}

\begin{document}

\title[]{On Sublevel Set Estimates and the Laplacian} 
\keywords{Sublevel set estimate, Hessian, Laplace operator, Brownian motion, Champagne domain, Potential Theory.}
\subjclass[2010]{26D10, 31A15, 35J05, 42B99} 

\thanks{S.S. is supported by the NSF (DMS-1763179) and the Alfred P. Sloan Foundation.}

\author[]{Stefan Steinerberger}
\address[Stefan Steinerberger]{Department of Mathematics, Yale University, New Haven, CT 06510, USA}
\email{stefan.steinerberger@yale.edu}

\begin{abstract} Carbery proved that if $u:\mathbb{R}^n \rightarrow \mathbb{R}$ is a positive, strictly convex function satisfying $\det D^2u \geq 1$, then we have the estimate 
$$ \left| \left\{x \in \mathbb{R}^n: u(x) \leq s \right\} \right| \lesssim_n s^{n/2}$$
and this is optimal. We give a short proof that also implies other results. Our main result is an estimate for the sublevel set of functions $u:[0,1]^2 \rightarrow \mathbb{R}$ satisfying $1 \leq \Delta u \leq c$ for some universal constant $c$: for any $\alpha > 0$, we have 
$$ \left| \left\{x \in [0,1]^2 : |u(x)| \leq \varepsilon\right\}\right| \lesssim_{c} \sqrt{\varepsilon} + \varepsilon^{\alpha - \frac12} \int_{[0,1]^2}{\frac{|\nabla u|}{|u|^{\alpha}} dx}.$$
For 'typical' functions, we expect the integral to be finite for $\alpha < 1$.
While Carbery-Christ-Wright have shown that no sublevel set estimates independent of $u$ exist, this result shows that for 'typical' functions satisfying $\Delta u \geq 1$, we expect the sublevel set to be $\lesssim \varepsilon^{1/2-}$.
We do not know whether this is sharp or whether similar statements are true in higher dimensions.
\end{abstract}
\maketitle

\section{Introduction}

\subsection{Introduction} Sublevel set estimates encapsulate the notion that `if a real-valued function $u$ has a large derivative, then it cannot spend too much time near any fixed value' \cite{carbery}. If $u:\mathbb{R} \rightarrow \mathbb{R}$ satisfies $u^{(k)} \geq 1$ for some integer $k \geq 2$, then it cannot be close to any constant for a long time (since, ultimately, it has to 'curve upward'). This is formalized in the van der Corput Lemma  
$$ \left|\left\{x \in \mathbb{R}: |u(x)| \leq t\right\}\right| \lesssim_k t^{\frac1k},$$
where the implicit constant is independent of $u$: the extremal case behaves, up to constants, essentially like the monomial $u(x) = x^k/k!$ (see \cite{akc, carb2, rogers} for explicit constants). These questions are classical \cite{akc, rogers, estein} and well understood in one dimension. The problem becomes a lot harder in higher dimensions \cite{carb2, carb4, gress, phong}. The seminal paper of Carbery, Christ \& Wright \cite{carb2} shows that if $u:[0,1]^n \rightarrow \mathbb{R}$ satisfies $D^{\beta} u \geq 1$ for some multi-index $\beta$, then there is a constant $\varepsilon>0$, depending only $n$ and $\beta$, such that
$$ \left|\left\{x \in [0,1]^n: |u(x)| \leq t\right\}\right| \lesssim_{\beta, n}  t^{\varepsilon}.$$
One could perhaps assume that $\varepsilon = 1/|\beta|$ but this is far from known, the problem seems very difficult and intimately connected to problems in combinatorics, see \cite{carb2}. The problem is even open in $n=2$ dimensions for the differential inequality
$$ \frac{\partial^2 u}{\partial x \partial y} \geq 1$$
for which it is known \cite{carb2} that
$$ \left|\left\{x \in [0,1]^2: |u(x)| \leq t\right\}\right| \lesssim_{}  \sqrt{t} \sqrt{\log{(1/t)}}$$
but where it is not known whether the logarithm is necessary. The related combinatorial problems have been studied in their own right \cite{katz, katz2}.

\subsection{Carbery's Sublevel Set Estimate.}  Carbery asked whether it is possible to replace the condition $D^{\beta} u \geq 1$ by a nonlinear condition. Motivated by the corresponding theory for oscillatory integral operators, the condition $\det D^2 u \geq 1$ seems like a natural first step, however, the example $u(x,y)= xy$ shows that some further conditions are required. 
\begin{theorem}[Carbery \cite{carbery}]  Let $K$ be a convex domain, let $u:K \rightarrow \mathbb{R}$ be strictly convex and satisfy $u \geq 0$ as well as
$$ \det D^2 u \geq 1.$$
Then, for any $s > 0$,
$$ \left| \left\{x \in K: u(x) \leq s \right\} \right| \lesssim_n s^{n/2},$$
where the implicit constant depends only on the dimension.
\end{theorem}
We refer to the original paper \cite{carbery} for other related statements of a similar type.
This is the optimal scaling: consider the function $u:\mathbb{R}^n \rightarrow \mathbb{R}$
$$ u(x) = a_1 x_1^2 + \dots + a_n x_n^2$$
for some positive real numbers $a_1, \dots, a_n > 0$ that satisfy
$ \det D^2 u = 2^n a_1 \dots a_n = 1.$
The function is strictly convex, the sublevel sets are ellipsoids and
$$ \left| \left\{x \in K: u(x) \leq s \right\} \right| \sim_n s^{n/2}.$$
The quantity $\det D^2 u$ is also invariant under affine transformations, something that is required for these types of statements to hold.
Carbery's proof, albeit short, is fairly nontrivial. One of our contributions is a simpler proof.

\subsection{The Laplacian.} One could also wonder whether similar results are possible for other differential operators. The Laplacian $\Delta$ is a natural starting point. This case has been analyzed by Carbery, Christ \& Wright \cite{carb2} who proved that such statements must necessarily fail. More precisely, they show 

\begin{proposition}[Proposition 5.2., Carbery-Christ-Wright \cite{carb2}] For $0<\delta < 1/2$, there exists $u_{} \in C^{\infty}([0,1]^2)$ such that $\Delta u \equiv 1$ on $(0,1)^2$ and
$$ \left| \left\{x \in [0,1]^2: |u(x)| \geq \delta \right\} \right| \leq \delta.$$
\end{proposition}
 The construction uses the Mergelyan theorem and is thus intimately connected to two dimensions. The main contribution of our paper is to show that such constructions are 'rare', in a certain sense, since $|\nabla u|$ has to rather large in regions where $|u|$ is small. We also prove the existence of a constant $c_n$ such that $ \left| \left\{x \in [0,1]^n: |u(x)| \geq c_n \right\} \right| \|u\|_{L^{\infty}} \geq c_n.$

\section{Main Results}
\subsection{Revisiting Carbery's estimate.}
A natural question is whether it is possible to weaken the assumptions in Carbery's estimate (this was also discussed by Carbery-Maz'ya-Mitrea-Rule \cite{carb3} by very different means). We note that, with the inequality of arithmetic and geometric mean, we obtain
$$  1 \leq \det D^2 u = \prod_{i=1}^{n}{\lambda_i(D^2 u)} \leq \frac{1}{n^n} \left( \sum_{i=1}^{n}{ \lambda_i(D^2 u)} \right)^n = \frac{ (\Delta u)^n}{n^n},$$
where we used strict convexity (positivity of the eigenvalues of the Hessian) to invoke the inequality.
One could thus wonder whether the weaker condition $\Delta u \geq 1$ by itself is sufficient and it is fairly easy to see that this is not the case: consider $u(x) = x_1^2 + \varepsilon x_2^2$, then the set
$\left\{x: u(x) \leq 1\right\}$ can be arbitrarily large if we make $\varepsilon$ sufficiently small. We also see that the shape of this sublevel set is rather eccentric and, as it turns out, this is necessary. For spherical sublevel sets, we can indeed
establish the desired result under weaker conditions.

\begin{proposition} Suppose $u:\left\{x \in \mathbb{R}^n: \|x\| \leq r\right\} \rightarrow \mathbb{R}$ satisfies $\Delta u \geq 1$, then
$$\max_{\|x\| \leq r}u - \min_{\|x\| \leq r}u \geq \frac{r^2}{2n}.$$
\end{proposition}
This result is a consequence of the maximum principle (and, as such, has presumably been used many times in the literature).
It implies a short proof of Theorem 1: consider the domain $$\Omega = \left\{x \in K: |u(x)| \leq s \right\}.$$ Since $\Omega$ is convex, by John's ellipsoid theorem it contains an ellipsoid $E \subset \Omega$ such that $|E| \sim_n |\Omega|$. Let us apply the diagonal volume-preserving affine transformation that maps $E$ to a ball. Since the condition $\det D^2 u \geq 1$ is affinely invariant, it is preserved. Then the arithmetic-geometric inequality implies $\Delta u \gtrsim_n 1$ on the ball and we can apply Proposition 2: since $u \geq 0$, we obtain 
$$ \| u\|_{L^{\infty}(K)} \gtrsim |\Omega|^{\frac{2}{n}}$$
 which is the desired result.
 We will give two proofs of Proposition 2: one simple and using only the maximum principle and one that is slightly more complicated that will set the stage for the arguments that play a role later in the paper.

 \subsection{Another sublevel set estimate.} We return to the problem of understanding functions $u:[0,1]^2 \rightarrow \mathbb{R}$ that satisfy $1 \leq \Delta u \leq c$ for some fixed constant $c$. As was shown by Carbery-Christ-Wright (see \S 1.3.), no classical sublevel set estimates (i.e. depending only on $\varepsilon$ but not on the function $u$) are possible. One could then wonder about sublevel set estimates that somehow depend on $u$. The Markov inequality can be written as
 $$ \left| \left\{x \in [0,1]^2: |u(x)| \leq \varepsilon \right\}\right| \leq \varepsilon^{\alpha} \int_{[0,1]^2}{ \frac{1}{|u(x)|^{\alpha}} dx}.$$
 Needless to say, this estimate is not exactly of great interest. It is obviously true for all functions and not just those that satisfy $1 \leq \Delta u \leq c$. 
The main result of this paper is a slightly different inequality. 
\begin{theorem} Assume  $u:[0,1]^2 \rightarrow \mathbb{R}$ satisfies $1 \leq \Delta u \leq c$. Then, for all $\alpha > 0$,
 $$ \left| \left\{x \in [0,1]^2: |u(x)| \leq  \varepsilon \right\}\right| \lesssim_c \sqrt{\varepsilon} + (2\varepsilon)^{\alpha - \frac12} \int_{[0,1]^2}{ \frac{|\nabla u|}{|u|^{\alpha}} dx},$$
 where the implicit constant depends only on $c$.
\end{theorem}
We strongly emphasize that it is not clear to us to which extent this result is best possible or whether possibly sharper estimates might exist. Another interesting problem is 
to understand whether the result does indeed depend on $c$. If so, then there might be a connection to so-called champagne domains from potential theory, we discuss the connection after the proof.\\

It is interesting that this curious estimate gives the sharp (up to the endpoint in $\alpha$) results for several different types of functions.
We consider the simple example $u(x_1, x_2) = x_1^2 + x_2^{2}$ where
$$ \left| \left\{x \in [0,1]^2: |u(x)| \leq \varepsilon \right\}\right| \lesssim \varepsilon^{}.$$
We have $|\nabla u| \lesssim \max(|x_1|, |x_2|^{})$ and obtain
$$  \int_{[0,1]^2}{ \frac{|\nabla u|}{|u|^{\alpha}} dx} \lesssim  \int_{[0,1]^2}{ \frac{ \max(|x_1|, |x_2|^{})}{\max(|x_1|^{2\alpha}, |x_2|^{2\alpha})|^{}} dx}
\lesssim \int_{[0,1]^2}{ \frac{1}{|x_1|^{2\alpha -1}} dx_1}
$$
which is finite up to $\alpha < 3/2$.
 We also emphasize that Theorem 2, when considering analytic functions, seems to connect to a type of inverse \L{}ojasiewicz inequality \cite{lo1, lo2, lo3, lo4, lo5, lo6}. We recall the inequality: if $f:\mathbb{R}^n \rightarrow \mathbb{R}$ is analytic in a neighborhood of the origin and $f(0) = 0$ and $\nabla f(0) = 0$, then there is an open neighborhood around the origin as well as two constants $c> 0$ and $\rho < 1$ such that
$$ \left|\nabla f(x)\right| \geq c \left|f(x)\right|^{\rho}.$$
Explicit estimates on $\rho$ are available when $f$ is a polynomial (in terms of $n$ and the degree, see \cite{lo2} and references therein). This interesting connection suggests the possibility of applications of Theorem 2 to polynomials or analytic functions.\\

Our proof is strictly two-dimensional (exploiting several geometric arguments that fail in higher dimensions). It is not clear to us whether and to which extent similar results could hold in higher dimensions, i.e. for $u:[0,1]^n \rightarrow \mathbb{R}$ satisfying $1 \leq \Delta u \leq c$. We believe that this could be quite interesting.

\subsection{A flatness estimate.} Let us again assume $\Delta u \geq 1$ on $[0,1]^n$. As discussed in \S 1.3. above, it is possible that
$$ \left| \left\{x \in [0,1]^n: |u(x)| \geq \varepsilon \right\}\right| \leq \varepsilon$$
for arbitrarily small values of $\varepsilon$. How does such a function look like? Applying Proposition 2 immediately shows that the set $\left\{x \in [0,1]^n: |u(x)| \leq \varepsilon \right\}$ cannot contain a ball of radius $2\sqrt{n} \sqrt{\varepsilon}$. At the same time, the condition $\Delta u \geq 1$ implies that there are no local maxima inside, therefore every connected component of
$$ \left\{x \in [0,1]^n: u(x) \geq \varepsilon\right\} \qquad \mbox{must necessarily touch the boundary.}$$
Figure 1 shows how such a function could possibly look like.
\begin{center}
\begin{figure}[h!]
\begin{tikzpicture}[scale=6]
\draw [thick] (0,0) -- (1,0) -- (1,1) -- (0,1) -- (0,0);
\draw [thick]  (0.7, 0.55) to[out=0, in=200] (1,0.8);
\draw [thick] (0.7, 0.55) to[out=20, in=180] (1,0.85);
\draw [thick]  (0.4, 0.3) to[out=340, in=180] (1,0.1);
\draw [thick] (0.4, 0.3) to[out=0, in=190] (1,0.15);
\draw [thick]  (0.5, 0.5) to[out=180, in=0] (0,0.55);
\draw [thick] (0.5, 0.5) to[out=180, in=0] (0,0.6);
\draw [thick] (0.3, 0.6) to[out=90, in=270] (0.3,1);
\draw [thick] (0.3, 0.6) to[out=95, in=270] (0.27,1);
\draw [thick] (0.5, 0.52) to[out=90, in=270] (0.5,1);
\draw [thick] (0.5, 0.52) to[out=95, in=270] (0.47,1);
\draw [thick] (0.2, 0.2) to[out=270, in=45] (0,0);
\draw [thick] (0.2, 0.2) to[out=280, in=45] (0,0.01);
\draw [dashed] (0.7, 0.4) circle (0.2cm);
\end{tikzpicture}
\caption{A sketch of what $\left\{x: u(x) \geq \varepsilon \right\}$ could look like: they connect to the boundary and they intersect every $\sim \sqrt{\varepsilon}-$ball.}
\end{figure}
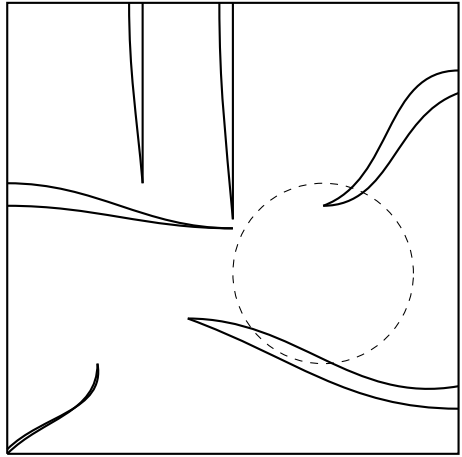
\end{center}
The Carbery-Christ-Wright construction shows that the set $\left\{x: u(x) \geq \varepsilon \right\}$ can indeed be arbitrarily small: we were interested in whether this required the function to be large in some places and this motivated our result: if solutions $\Delta u \geq 1$ are flat on a very large subset, then $u$ must be very large on the complement.

\begin{theorem} There exists a constant $c_n > 0$ depending only on the dimension such that if $u:[0,1]^n \rightarrow \mathbb{R}$ satisfies $\Delta u \geq 1$, then
$$ \left|\left\{x \in [0,1]^n: |u(x)| \geq c_n \right\} \right| \cdot \| u\|_{L^{\infty}([0,1]^n)}^{} \geq c_n.$$
\end{theorem}
It is not at all clear how the extremal examples behave, whether Theorem 3 has the optimal scaling. Is there an estimate like $$ \left|\left\{x \in [0,1]^n: |u(x)| \geq c_n \right\} \right| \cdot \| u\|_{L^{p}([0,1]^n)}^{\alpha} \geq c_n$$ for some $\alpha < 1$ or $p < \infty$ or both?

\section{Proofs of Proposition 2}

\subsection{First Proof of Proposition 2}
\begin{proof} The estimate is invariant under addition of constants. We can thus, by adding $-\min_{\|x\| \leq r} u$, assume that
$u \geq 0$ and it suffices to show that
$$\max_{\|x\| \leq r} u \geq \frac{r^2}{2n}.$$

We define the function $w$ as the solution of
\begin{align*}
 \Delta w &= 1 \qquad \quad \,\,\, \mbox{for}~\|x\| < r\\
w &= \max_{\|x\| = r}u  \quad \mbox{for}~\|x\| = r.
\end{align*}
We see that $w - u$ is positive on the boundary and that
$$\Delta(w - u) = \Delta w - \Delta u \leq 0.$$
This implies that the minimum of $w - u$ is assumed on the boundary, where the function is nonnegative. Therefore $w \geq u \geq 0$. However, we can actually compute $w$ in closed form: the solution is radial and the radial Laplacian can be written as
$$ \Delta_s = \frac{1}{s^{n-1}} \frac{\partial}{\partial s}\left( s^{n-1} \frac{\partial f}{\partial s} \right)$$
which shows that the solution is given by
$$ w(s) = \frac{s^2}{2n} + c,$$
where $c$ is a constant chosen such that the boundary conditions are satisfied. However, $w \geq 0$ and thus $c \geq 0$. This implies that
$$  \max_{\|x\| = r}u = w(r) \geq \frac{r^2}{2n}.$$
It is clear that the estimate is sharp since
$$ \Delta\left( \frac{\|x\|^2}{2n}\right) = \Delta \left( \frac{x_1^2 + \dots + x_n^2}{2n}\right) = 1.$$
\end{proof}

\subsection{Second Proof of Proposition 2}
\begin{proof} Our second proof of Proposition 2 is based on representing the function $u$ as the stationary solution of the heat equation. By itself, this argument is more difficult than the one based on the maximum principle but it introduces a line of thought that will be useful for a later proof (indeed, this type of argument has proven useful for several different problems \cite{biswas, lierl, manas, stein, stein2}). We study 
\begin{align*}
\frac{\partial v}{\partial t} &= \Delta v - \Delta u \qquad \,\,\,\mbox{in}~\Omega \\
v(0,x) &= u(x)  \qquad \mbox{in}~\Omega \\
v(t,x) &= u(x) \qquad \mbox{on}~ \partial \Omega.
\end{align*}
The Feynman-Kac formula then implies a representation of the function $u(x) = v(t,x)$ as a weighted average of its values in a neighborhood to which standard estimates can be applied. We denote a
Brownian motion started in $x \in \Omega$ at time $t$ by $\omega_x(t)$. The Dirichlet boundary conditions require us to demand that
the boundary is 'sticky' and that a particle remains at the boundary once it touches it. The Feynman-Kac formula implies that for all $t > 0$
$$ u(x) = \mathbb{E} u(\omega_x(t)) - \mathbb{E} \int_{0}^{t \wedge \tau}{ (\Delta u)(\omega_x(t)) dt},$$
where $\tau_x$ is the stopping time for impact on the boundary for Brownian motion started in $x$. A simple way to derive the scaling for Proposition 2 is now as follows: suppose 
we are on a ball of radius $r$.
We have, for all $\|x\| < r$ and all $t > 0$,
$$ \left| u(x) - \mathbb{E} u(\omega_x(t)) \right| \leq \max_{\|x\| \leq r}u - \min_{\|x\| \leq r}u.$$
The expected lifetime of Brownian motion in a ball of radius $r$ when started near the center until hitting the boundary is $\sim r^2$. Since $\Delta u \geq 1$ this implies
$$ \mathbb{E} \int_{0}^{t \wedge \tau}{ (\Delta u)(\omega_x(t)) dt} \gtrsim r^2$$
and this establishes the desired result.
\end{proof}
One obvious advantage of this kind of approach is that $\Delta u \geq 1$ is clearly not strictly required as long as it is true 'in the aggregate'. Indeed, if we have $\Delta u \geq \phi(x)$, then approaches of this flavor could be used to deduce analogous bounds as long as $'\phi \geq 1'$ in a suitable averaged sense.

\section{Proof of Theorem 2}
\subsection{Introduction.} We start with a quick variation on Proposition 2 that will prove useful in the proof of Theorem 2 where we require a pointwise statement.
\begin{proposition} Suppose $u:\left\{x \in \mathbb{R}^n: \|x\| \leq r\right\} \rightarrow \mathbb{R}$ satisfies $\Delta u \geq 1$, then, for all $\|y\| < r$
$$\max_{\|x\| = r}u(x)  \geq \frac{r^2 - \|y\|^2}{2n} + u(y).$$
\end{proposition}
\begin{proof} The proof is almost completely analogous. 
We define the function $w$ as the solution of
\begin{align*}
 \Delta w &= 1 \qquad  \quad \,\, \mbox{for}~\,~\|x\| < r\\
w &= \max_{\|x\| = r}u  \quad \mbox{for}~\|x\| = r
\end{align*}
and observe that, as before, $w \geq u$. Moreover, we know that
$$ w(r) = \frac{r^2}{2n} + c$$
for some constant $c \geq 0$ chosen so that the boundary conditions are satisfied. The argument is finished by observing that
\begin{align*}
 \frac{r^2}{2n} - \frac{\|y\|^2}{2n} &= \left(  \frac{r^2}{2n} +c \right) - \left(\frac{\|y\|^2}{2n} + c \right)\\
 &=   \max_{\|x\|=r}u(x) - w(y)  \leq \max_{\|x\|=r}u(x) - u(y).
 \end{align*}
\end{proof}
We will use this statement to conclude, for all $y$ and all $r>0$,
$$ \max_{\|x-y\|=r} u(x) - u(y) \gtrsim_n r^2.$$
The proof of Theorem 2 uses the coarea formula which we recall for the convience of the reader. If $\Omega \subset \mathbb{R}^n$ is an open set, $u:\Omega \rightarrow \mathbb{R}$ is Lipschitz and $g \in L^1(\Omega)$, then
$$ \int_{\Omega}{g(x) |\nabla u(x)| dx} = \int_{\mathbb{R}}{ \left( \int_{u^{-1}(t)} g(x) d \mathcal H^{n-1} \right) dt}.$$
We will use this identity for the function
$$ g(x) = \frac{\chi_{\varepsilon \leq |u(x)| \leq 2\varepsilon}}{|u|^{\alpha}}$$
which is bounded and thus in $L^1$. \\

\subsection{The Proof.} We will describe the proof of Theorem 2 which nicely decouples into several different steps. We annotate the steps and give precise details below. The coarea formula and a pigeonhole principle imply (Step 1) the existence of
$$ -2 \varepsilon < t_1 < - \varepsilon  < 0 < \varepsilon < t_2 < 2\varepsilon$$
such that, for $i=1,2$,
$$ \mathcal{H}^1\left( \left\{x: u(x) = t_i\right\}\right) \leq (2\varepsilon)^{\alpha - 1} \int_{[0,1]^2}{ \frac{|\nabla u|}{|u|^{\alpha}} dx}.$$
The next step is topological: we argue that, for all $t \in \mathbb{R}$, the set
$$ \left\{x \in [0,1]^2: u(x) \leq t\right\} \qquad \mbox{is a union of simply connected components.}$$
Let us take such a connected component and assume it is not simply connected. Then there is at least one `hole`, where $u$ is bigger than $t$. This means that $u$ has a local maximum which contradicts the assumption that $\Delta u \geq 1$. We now fix these values and introduce a shift of the original function 
$$ v(x) = u(x) - t_2\quad \mbox{and estimate the size of} \quad \left\{ x \in [0,1]^2: 0 \leq v(x) \leq t_2 - t_1\right\}.$$
The set $\left\{x: v(x) \leq t_2 - t_1\right\}$ is the union of simply connected sets. We fix such a simply connected set $A$. We observe that
$$ \left\{ x \in A: v(x) \leq 0\right\} \qquad \mbox{is the union of simply connected sets (see Fig. 2).}$$

\begin{center}
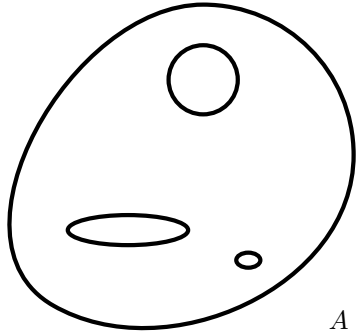
\begin{figure}[h!]
\begin{tikzpicture}[scale=2]
\draw [ultra thick] (0,0) to[out=330, in=270] (2,1) to [out=90, in =0] (1,2) to[out=180, in=150] (0,0);
\draw [ultra thick] (1,1.5) circle (0.23cm);
\draw [ultra thick] (1.3,0.3) ellipse (0.08cm and 0.05cm);
\draw [ultra thick] (0.5,0.5) ellipse (0.4cm and 0.1cm);
\node at (1.9, -0.1) {$A$};
\end{tikzpicture}
\caption{A connected component $A$ and three simply connected subsets where $v$ is smaller than 0: two of them are big one, is small.}
\end{figure}
\end{center}

We label the connected components of $\left\{ x \in A: v(x) \leq 0\right\}$ by $C_1, C_2, \dots$ and put them into two different categories: the ones for which $\mathcal{H}^1(\partial C_i) \geq \sqrt{\varepsilon}$
and the `small' ones for which  $\mathcal{H}^1(\partial C_i) < \sqrt{\varepsilon}$. If $\mathcal{H}^1(\partial C_i) < \sqrt{\varepsilon}$, then (Step 2)
we will show that
$$ \min_{x \in C_i} v(x) \geq - 4 \varepsilon \| \Delta u\|_{L^{\infty}}.$$
We define $\Omega$ as the set $A$ with all `small' connected components of $\left\{x: v(x) \leq 0\right\}$ included and all large connected components removed
$$ \Omega = A \setminus \bigcup_{\mathcal{H}^1(\partial C_i) \geq \sqrt{\varepsilon}}{C_i}.$$
We observe that, for all $x \in \Omega$,
$$ - 4 \|\Delta\|_{L^{\infty}} \varepsilon \leq v(x) \leq t_2 - t_1 \leq 4 \varepsilon.$$
We will then show (Step 3) that this has an interesting implication: let $B_x(t)$ be a Brownian motion started in $x$ at time $t$. We assume that it gets absorbed when hitting the boundary $\partial \Omega$
and we use $\tau_x$ to denote the first hitting time. Then, for all $x \in \Omega$,
$$ \mathbb{E}~\tau_x \leq  (4\|\Delta\|_{L^{\infty}} + 4) \varepsilon.$$
This means that the domain $\Omega$ is very interesting: no matter where one starts, Brownian motion is likely to hit the boundary rather quickly.
This has a physical interpretation: if we assume to have temperature 0 throughout $\Omega$ and have fixed boundary temperature 1 at the boundary and assume $\phi(t,x)$ to denote the temperature
in the point $x$ at time $t$, then, by Markov's inequality,
$$ \phi(  (8\|\Delta\|_{L^{\infty}} + 8) \varepsilon , x) \geq \frac{1}{2}.$$
Put differently, if Brownian motion hits the boundary quickly, then the domain is heating up quickly. The argument then follows from using heat content estimates (Step 4): we will show that
$$ \frac{1}{2} | \Omega| \lesssim \int_{\Omega}{  \phi(  (8+8\|\Delta\|_{L^{\infty}}) \varepsilon, x) dx} \lesssim \|\Delta\|_{L^{\infty}}^{1/2} \cdot  \sqrt{ \varepsilon} \cdot |\partial \Omega|,$$
where the implicit constants are universal. Summing over all connected domains, we observe that 
$$ | \left\{ x \in [0,1]^2: 0 \leq v(x) \leq t_2 - t_1\right\}| \lesssim  \|\Delta\|_{L^{\infty}}^{1/2} \cdot \sqrt{ \varepsilon} \cdot \sum_{i=1}^{2} \mathcal{H}^1\left( \left\{x: u(x) = t_i\right\}\right).$$
However, by construction, we have 
 $$ \sum_{i=1}^{2} \mathcal{H}^1\left( \left\{x: u(x) = t_i\right\}\right) \leq (2\varepsilon)^{\alpha - 1} \int_{[0,1]^2}{ \frac{|\nabla u|}{|u|^{\alpha}} dx}$$
 and this concludes the argument. A small addendum is that we need to be careful with regards to domains that touch the boundary since then some of the estimates will
 need to be modified (Step 5) and this accounts for the additional $\sim_c \sqrt{\varepsilon}.$

\subsection{Details of the Argument.} We now describe the details.\\

\textbf{Step 1.}
Let us fix $\varepsilon > 0$. 
We use the coarea formula to estimate
\begin{align*}
\int_{[0,1]^2}{ \frac{|\nabla u|}{|u|^{\alpha}} dx} &\geq \int_{[0,1]^2}{ \frac{|\nabla u|}{|u|^{\alpha}} \chi_{\varepsilon \leq u(x) \leq 2\varepsilon} dx} \\
&=  \int_{\varepsilon < t < 2\varepsilon}{ \frac{\mathcal{H}^1\left( \left\{x: u(x) = t\right\}\right)}{|t|^{\alpha}} dt}
\end{align*}
and thus
$$ \frac{1}{\varepsilon}  \int_{\varepsilon < t < 2\varepsilon}{ \frac{\mathcal{H}^1\left( \left\{x: u(x) = t\right\}\right)}{|t|^{\alpha}} dt} \leq \varepsilon^{-1} \int_{[0,1]^2}{ \frac{|\nabla u|}{|u|^{\alpha}} dx}$$
and therefore there exists a $\varepsilon \leq t \leq 2\varepsilon$ such that
$$ \mathcal{H}^1\left( \left\{x: u(x) = t\right\}\right) \leq (2\varepsilon)^{\alpha - 1} \int_{[0,1]^2}{ \frac{|\nabla u|}{|u|^{\alpha}} dx}.$$
The same argument can be applied to prove the existence of a level set $\left\{x: u(x) = t_2\right\}$ for some $-2\varepsilon < t_2 < - \varepsilon$ for
which 
$$ \mathcal{H}^1\left( \left\{x: u(x) = t\right\}\right) \leq (2\varepsilon)^{\alpha - 1} \int_{[0,1]^2}{ \frac{|\nabla u|}{|u|^{\alpha}} dx}.$$\

\textbf{Step 2.} Step 2 boils down to the following Lemma.

\begin{lem}
Let $D \subset \mathbb{R}^2$ be a simply connected domain and let $\psi:D \rightarrow \mathbb{R}$ be a function satisfying $1 \leq \Delta \psi \leq c$ that vanishes on the boundary $\partial D$. Then
$$ \min_{x \in D}{\psi(x)} \geq - 4  c \cdot \mathcal{H}^1(\partial D)^2$$
\end{lem}
\begin{proof} There are different ways of proving this statement. We return to the probabilistic interpretation and introduce the equation
\begin{align*}
 \frac{\partial w}{\partial t}  &= \Delta w - \Delta \psi  \qquad \mbox{inside}~\Omega\\
 w &= \psi \qquad \qquad \quad \quad \mbox{on}~\partial \Omega.
\end{align*}
We observe that $\psi$ is the stationary-in-time solution. In particular, the Feynman-Kac formula implies that $\psi$ satisfies
$$ \psi(x) = \mathbb{E} ~\psi(\omega_x(t)) - \mathbb{E} \int_{0}^{t \wedge \tau_x}{ (\Delta \psi)(\omega_x(t)) dt},$$
where $\omega_x(t)$ is Brownian motion started in $x$, running until time $t$ and getting absorbed by the boundary upon impact. The boundary hitting time 
for a Brownian motion started in $x$ is denoted by $\tau_x$. We let $t \rightarrow \infty$ so that things simplify to
$$ \psi(x) =  \mathbb{E} \int_{0}^{ \tau_x}{ -(\Delta \psi)(\omega_x(t)) dt} \geq -c \cdot \mathbb{E} ~\tau_x.$$
The proof is then completed by a standard estimate (see e.g. Ba\~nuelos \& Carroll \cite{ban}) which states that
$$ \sup_{x \in D} ~\mathbb{E} ~\tau_x  \leq 4 \cdot \mbox{inrad}(D)^2.$$
Weaker geometric bounds suffice for our purpose, we have 
$$ \mbox{inrad}(D)^2 \leq |D| \leq  \mathcal{H}^1(\partial D)^2$$
which implies the result.
\end{proof}

\textbf{Step 3.} Step 3 follows from the following Lemma.

\begin{lem}
Let $\Omega \subset \mathbb{R}^2$ be a bounded planar domain and assume that $\psi:\Omega \rightarrow \mathbb{R}$ satisfies
$$ \forall x\in \Omega: \qquad a \leq \psi(x) \leq b$$
as well as $\Delta \psi \geq 1$. Then, for any $x \in \Omega$, the first hitting time $\tau_x$ of Brownian motion started in $x$, satisfies
$$ \mathbb{E}~\tau_x \leq b-a.$$
\end{lem}
\begin{proof}
We again work with the equation
\begin{align*}
 \frac{\partial w}{\partial t}  &= \Delta w - \Delta \psi  \qquad \mbox{inside}~\Omega\\
 w &= \psi \qquad \qquad \quad \quad \mbox{on}~\partial \Omega.
\end{align*}
and the associated stochastic interpretation
$$ \psi(x) = \mathbb{E} ~\psi(\omega_x(t)) - \mathbb{E} \int_{0}^{t \wedge \tau_x}{ (\Delta \psi)(\omega_x(t)) dt}.$$
We rewrite this equation as
\begin{align*}
 \mathbb{E}~t \wedge \tau_x &\leq  \mathbb{E} \int_{0}^{t \wedge \tau_x}{ (\Delta \psi)(\omega_x(t)) dt} \\
&\leq \mathbb{E} ~\psi(\omega_x(t)) - \psi(x)\\
&\leq b-a.
\end{align*}
The result now follows from letting $t \rightarrow \infty$.
\end{proof}

\textbf{Step 4.} It remains to derive the relevant heat content estimate. We note that these estimates are not new and various sharper
forms have been derived (see e.g. \cite{van} and references therein). We will only need a simple specialized estimate which can be developed
from scratch.

\begin{center}
\begin{figure}[h!]
\begin{tikzpicture}
\filldraw [ultra thick] (1,1) circle (0.13cm);
\filldraw [ultra thick] (0,0) ellipse (0.5cm and 0.03cm);
\filldraw [ultra thick] (3,0) circle (0.15cm);
\draw [ultra thick] (-1,1.7) to[out=40, in = 180] (4, 1) to[out=0, in=10] (4,-1);
\draw [ultra thick] (-2, -0.5) to[out=330, in=190] (4,-1) ;
\draw [ultra thick] (-2, -0.5) to[out=150, in=220] (-1,1.7) ;
\node at (1,-0.8) {$\Omega$};
\end{tikzpicture}
\caption{Consider such a set $\Omega$: a nice simply connected domain with several subdomains removed. Assuming that all subdomains have large boundary $\mathcal{H}^1(\partial C_i) \geq \sqrt{\varepsilon}$, how quickly can such a domain be heated up within $\varepsilon$ units of time?}
\end{figure}
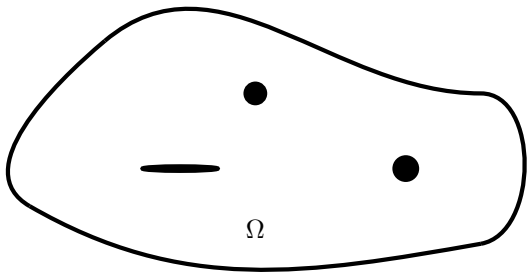
\end{center}
The question is geometric (see Fig. 3). We are given a simply connected domain and remove several simply connected sets with large boundary $\mathcal{H}^1(\partial C_j) \geq \varepsilon$. The question
is: if this domain $\Omega$ is initially at temperature 0 and we have constant temperature 1 on the boundary, what is the total temperature in the room at time $t$?
More formally, we introduce $\phi(t,x)$ as the solution of
\begin{align*}
 \frac{\partial \phi}{\partial t}  &= \Delta \phi   \qquad \mbox{inside}~\Omega\\
 \phi(0,x) &= 0   \qquad \quad \mbox{inside}~\Omega\\
 \phi(t,x) &= 1 \qquad  \quad \mbox{on}~\partial \Omega.
\end{align*}
We are interested in estimates on
$$ \int_{\Omega}{ \phi(t,x) dx} \qquad \mbox{from above}.$$
Moreover, these estimates are supposed to be uniform in everything since we have little control over the domain and no control over $\varepsilon$.

\begin{lem}
We have, for sets of the type under consideration and for some universal implicit constant,
$$ \int_{\Omega}{ \phi(\varepsilon,x) dx} \lesssim \sqrt{\varepsilon} \cdot | \partial \Omega|.$$
\end{lem}
\begin{proof} There are two sources of heat: the `big' boundary of $\Omega$ that encloses everything and the boundaries of the `large' sets that were removed. All of them contribute to heating
the room. However, their effects are subadditive. This can be easily seen by the probabilistic interpretation: if we start Brownian motion in a point $x$ then the likelihood of hitting an obstacle
$A$ or hitting $B$ is less than the likelihood of hitting $A$ in isolation or $B$ in isolation (if only $A$ is present, then there are paths that go through where $B$ would have been and still end
up hitting $A$). Of course, if the obstacles are tiny, then the effect is \textit{almost} additive. In the language of probability theory, this is merely the union bound.\\

Let us fix a point $x$ and let us assume that $C_1, \dots, C_n, \dots$ denotes the simply connected domains that were removed (whose boundary length is at least $\sqrt{\varepsilon}$). 
We define a series of events $A, A_1, \dots, A_n, \dots$ where $A$ is the outmost boundary of $\Omega$ (the connected component of the level set) and
$$ A_k(x,t) = \mathbb{P}\left(\mbox{Brownian motion started in $x$ hits $C_k$ within $t$ units of time}\right).$$
We note that the convention here is that the Brownian motion gets absorbed upon impact with the boundary (so it cannot hit $C_1$ and then $C_3$, in particular it can hit at most one $C_i$).
We can then interpret the heat equation as a probability and apply the union bound
$$ \phi(t,x) = \mathbb{P}\left( A(t,x) \cup \bigcup_{k=1}^{\infty}{A_k(t,x)} \right) \leq \mathbb{P}(A(t,x)) + \sum_{k=1}^{\infty}{\mathbb{P}(A_k(t,x))}.$$
Let us now fix $A_j$ and obtain an upper bound for it (the argument will also apply to $A$). We estimate it further from above by assuming that $\Omega$ does not even exist: we assume
$C_j$ lies isolated in $\mathbb{R}^2$. Fix some $x \in \mathbb{R}^2$ such that $x \notin C_j$. We assume w.l.o.g. that $x$ lies on the real line and that the same is true for the point on $C_j$
that is closest to $x$. The likelihood of a Brownian motion started in $x$ hitting $C_j$ within $t$ units of time can be bounded from above by the likelihood of the $x-$coordinate of the Brownian
motion (itself a Brownian motion) traveling at least distance $d(x,  C_j)$. 
The reflection principle implies that the one-dimensional Brownian motion $B(t)$ satisfies
\begin{align*}
\mathbb{P} \left(\max_{0 \leq s \leq t} B(s) \geq d(x,  C_j)\right) &= 2\cdot \mathbb{P}{\left( B(t) \geq d(x, C_j)\right)}\\
&= \frac{2}{\sqrt{2 \pi t}} \int_{d(x,C_j)}^{\infty}{ \exp\left(-\frac{y^2}{2t}\right)  dy}.
\end{align*}
Thus we have obtained the bound
$$ \mathbb{P}(A_j(x,t)) \leq \frac{2}{\sqrt{2 \pi t}} \int_{d(x,C_j)}^{\infty}{ \exp\left(-\frac{y^2}{2t}\right)  dy}.$$
Of course, we also have the trivial estimate $ \mathbb{P}(A_j(x,t))  \leq 1$. We will now apply both these estimates in different regimes. We use the trivial estimate  $ \mathbb{P}(A_j(x,t))  \leq 1$ inside the $\sqrt{t}-$neighborhood
of $C_j$. 
\begin{align*}
 \int_{\mathbb{R}^2} \mathbb{P}(A_j(x,t)) dx &\lesssim  |\left\{x: d(x,C_j) \leq \sqrt{t} \right\}| +  \int_{d(x, C_j) \geq \sqrt{t}} \mathbb{P}(A_j(x,t)) dx
\end{align*}
We bound the second integral by
\begin{align*}
\int_{d(x, C_j) \geq \sqrt{t}} \mathbb{P}(A_j(x,t)) dx &\leq \int_{d(x, C_j) \geq \sqrt{t}}  \frac{2}{\sqrt{2 \pi t}} \int_{d(x,C_j)}^{\infty}{ \exp\left(-\frac{y^2}{2t}\right)  dy} dx
\end{align*}
Since this upper bound depends only on the distance, we can use the coarea formula to change coordinates and bound this integral, denoted by $J$, from above by
$$ J \leq   \int_{\sqrt{t}}^{\infty} \mathcal{H}^1\left( \left\{x: d(x,C_j) = s\right\}\right)  \frac{2}{\sqrt{2 \pi t}} \int_{s}^{\infty}{ \exp\left(-\frac{y^2}{2t}\right)  dy}$$
This is now where isoperimetry comes into play. We note an upper bound on the length of an exterior parallel curve (see e.g. Bandle \cite{band} or Do Carmo \cite{do})
$$ \mathcal{H}^1\left( \left\{x: d(x,C_j) = s\right\}\right) - \mathcal{H}^1(\partial C_j) \leq 2\pi s.$$
Therefore, we can bound
\begin{align*}
 J   &\leq   \int_{\sqrt{t}}^{\infty} \mathcal{H}^1\left( \left\{x: d(x,C_j) = s\right\}\right)  \frac{2}{\sqrt{2 \pi t}} \int_{s}^{\infty}{ \exp\left(-\frac{y^2}{2t}\right)  dy}\\
 &\lesssim  \int_{\sqrt{t}}^{\infty} \mathcal{H}^1\left( \partial C_j \right)  \frac{2}{\sqrt{2 \pi t}} \int_{s}^{\infty}{ \exp\left(-\frac{y^2}{2t}\right)  dy} \\
 &+ \int_{\sqrt{t}}^{\infty}  \frac{2 s}{\sqrt{2 \pi t}} \int_{s}^{\infty}{ \exp\left(-\frac{y^2}{2t}\right)  dy}.
\end{align*}
We now specialize $t = \varepsilon$. Then the first integral simplifies to
$$  \int_{\sqrt{\varepsilon}}^{\infty} \mathcal{H}^1\left( \partial C_j \right)  \frac{2}{\sqrt{2 \pi \varepsilon}} \int_{s}^{\infty}{ \exp\left(-\frac{y^2}{2\varepsilon}\right)  dy} \lesssim \sqrt{\varepsilon} \cdot  \mathcal{H}^1\left( \partial C_j \right).$$
The second integral simplifies to, making use of $\mathcal{H}^1(\partial C_j) \geq \sqrt{\varepsilon}$,
$$ \int_{\sqrt{e}}^{\infty}  \frac{2 s}{\sqrt{2 \pi e}} \int_{s}^{\infty}{ \exp\left(-\frac{y^2}{2e}\right)  dy} \lesssim \varepsilon \leq  \sqrt{\varepsilon} \cdot \mathcal{H}^1\left( \partial C_j \right).$$
\end{proof}

 \textbf{Step 5.} We now finally collect all the results in one place. We are given the domain 
 $$ \Omega = A \setminus \bigcup_{\mathcal{H}^1(\partial C_i) \geq \sqrt{\varepsilon}}{C_i} \subset \Omega$$
 for which we have estimates on $\mathcal{H}^1(\partial \Omega \cap (0,1)^2)$. This, implies, by adding the boundary of the unit square, that
 $$\mathcal{H}^1(\partial \Omega) \leq 4 + \mathcal{H}^1(\partial \Omega \cap (0,1)^2).$$

 Our knowledge about properties of $\Omega$ is implicit: we know that for any point $x \in \Omega$, a Brownian motion
 started in $x$ hits the boundary with an expected hitting time of 
 $$ \mathbb{E}~\tau_x \leq  (4\|\Delta\|_{L^{\infty}} + 4) \varepsilon.$$
 This means that if we start the heat equation with zero initial conditions in $\Omega$ and constant boundary conditions 1 on $\partial \Omega$, then the solution
 of the heat equation is uniformly bounded from below by $1/2$ at time
 $$ t=  (8\|\Delta\|_{L^{\infty}} + 8) \varepsilon.$$
 However, we also have an estimate telling us that the total amount of heat inside the domain satisfies
 $$ \int_{\Omega}{ \phi(\varepsilon,x) dx} \lesssim \sqrt{\varepsilon} \cdot | \partial \Omega|.$$
 This implies the desired result.
 
 \subsection{Champagne subdomains.} One natural question is whether the upper bound in our assumption $1 \leq \Delta u \leq c$ is necessary or whether it could be weakened. Our proof
 gives a rather precise characterization how `bad' examples could look like. We only used $\Delta u \leq c$ in one spot in the proof: to conclude that simply connected subsets of $\left\{x: u(x) \leq 0\right\}$,
whose boundary is shorter than $\sqrt{\varepsilon}$ only contain function values that are a tiny bit smaller than 0. If the Laplacian can be arbitrarily positive, then the function could assume a lot
smaller values inside those domains.
 
 \begin{center}
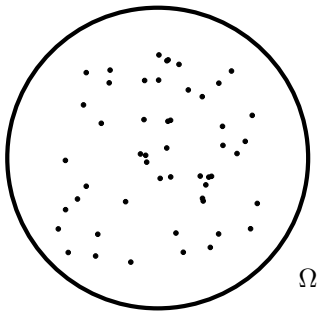
\begin{figure}[h!]
\begin{tikzpicture}[scale=2]
\draw [ultra thick] (0,0) circle (1cm);
\foreach \x in {0,...,50}{
    \filldraw (rand*0.7,0.7*rand) circle (0.015cm);
}
\node at (1,-0.8) {$\Omega$};
\end{tikzpicture}
\caption{A nice domain $\Omega$ with many small subdomains removed.}
\end{figure}
\end{center}

However, we still have the property that Brownian motion started in any point has a hitting time whose expectation is uniformly bounded by $\varepsilon$. This means that these `bubbles' need
to nicely distributed throughout the domain. This problem has actually been studied in the literature where such domains are also known as champagne domains. We refer to \cite{chm2, chm1, ch0, ch00, ch000, ch1, ch2}. One natural question now is: are the functions satisfying $\Delta u \geq 1$ extremal when they have $\Delta u$ huge in inside the bubbles of a champagne domain?

\subsection{Proof of Theorem 3}

\begin{proof} We want to prove the existence of a constant $c_n$, depending only on $n$, such that for all $u:[0,1]^n \rightarrow \mathbb{R}$ satisfying $\Delta u \geq 1$, we have
$$ \left|\left\{x \in [0,1]^n: |u(x)| \geq c_n \right\} \right| \cdot \| u\|_{L^{\infty}}^{} \geq c_n.$$
We will prove it for the unit ball instead of the unit cube
$$ \left|\left\{\|x\| \leq 1: |u(x)| \geq c_n \right\} \right| \cdot \| u\|_{L^{\infty}(\left\{x:\|x\| \leq 1\right\})}^{} \geq c_n$$
which implies the original result since the unit cube contains a ball of radius $1/2$ and we only work up to constants depending only on the dimension.
 Now suppose the desired statement is false. Then, for any $\varepsilon > 0$, there exists a function satisfying $\Delta u \geq 1$ for which
$$ \left|\left\{\|x\| \leq 1: |u(x)| \geq \varepsilon \right\} \right| \cdot \| u\|_{L^{\infty}}^{} \leq \varepsilon.$$
We will now argue for a fixed (but unspecified) value of $\varepsilon$ and will then see that $\varepsilon$ cannot be chosen arbitrarily small.
 Fubini's theorem combined with the pigeonhole principle implies the existence of a $0.99 < t < 1$ for which
$$ \mathcal{H}^{n-1}\left(  \left\{\|x\| = t: |u(x)| \geq \varepsilon \right\} \right) \cdot \|u\|_{L^{\infty}(\|x\| \leq 1)} \lesssim_n \varepsilon.$$
We fix this value of $t$.

\begin{center}
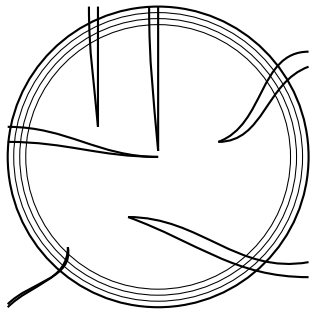
\begin{figure}[h!]
\begin{tikzpicture}[scale=4]
\draw [thick] (0.5,0.5) circle (0.5cm);
\draw [] (0.5,0.5) circle (0.48cm);
\draw [] (0.5,0.5) circle (0.46cm);
\draw [] (0.5,0.5) circle (0.44cm);
\draw [thick]  (0.7, 0.55) to[out=0, in=200] (1,0.8);
\draw [thick] (0.7, 0.55) to[out=20, in=180] (1,0.85);
\draw [thick]  (0.4, 0.3) to[out=340, in=180] (1,0.1);
\draw [thick] (0.4, 0.3) to[out=0, in=190] (1,0.15);
\draw [thick]  (0.5, 0.5) to[out=180, in=0] (0,0.55);
\draw [thick] (0.5, 0.5) to[out=180, in=0] (0,0.6);
\draw [thick] (0.3, 0.6) to[out=90, in=270] (0.3,1);
\draw [thick] (0.3, 0.6) to[out=95, in=270] (0.27,1);
\draw [thick] (0.5, 0.52) to[out=90, in=270] (0.5,1);
\draw [thick] (0.5, 0.52) to[out=95, in=270] (0.47,1);
\draw [thick] (0.2, 0.2) to[out=270, in=45] (0,0);
\draw [thick] (0.2, 0.2) to[out=280, in=45] (0,0.01);
\end{tikzpicture}
\caption{If a set $\left\{x: u(x) \geq c_n\right\}$ inside a ball is small, then there also exists a slight shrinking of the ball, such that the $(n-1)-$dimensional size of the set $ \cap \left\{x: \|x\|=t\right\}$ is small.}
\end{figure}
\end{center}

The next step is to argue as in the (second) proof of Proposition 2. We rewrite the function $u$ as the stationary solution of a heat equation and obtain the equation
$$ u(x) = \mathbb{E} u(\omega_x(t)) - \mathbb{E} \int_{0}^{t \wedge \tau}{ (\Delta u)(\omega_x(t)) dt}.$$
We now let $t \rightarrow \infty$. In that regime, all the Brownian motion particles are impacted on the boundary and we can reinterpret
$\mathbb{E} u(\omega_x(t))$ as an integral over the boundary with respect to harmonic measure. We integrate this identity in the ball
$\left\{\|x\| \leq 1/100\right\}$. The symmetry of the ball (and the inherited symmetry of the harmonic measure) implies, for some positive constants $c_{n,1}, c_{n,2} > 0$ that only depend on the dimension (and, very mildly and in a way that can be controlled, on $t$)
\begin{align*}
 \int_{\|x\| \leq 1/100}{ u(x) dx} &=c_{n,1} \int_{\|x\|=t}{ u(x) d \mathcal{H}^{n-1}} \\
 &-  c_{n,2} \int_{\|x\| \leq 1/100}{  \mathbb{E} \int_{0}^{\tau}{ (\Delta u)(\omega_x(t)) dt}dx}.
 \end{align*}
The argument can now be concluded as follows: the first integral is certainly small since
$$  \int_{\|x\| \leq 1/100}{ u(x) dx} \leq   \varepsilon + \left|\left\{|u(x)| \geq \varepsilon \right\}\right| \cdot \| u\|_{L^{\infty}}^{} \leq 2 \varepsilon.$$
The second integral is also small. Ignoring the constant in front, which only depends on the dimension, we can estimate
$$\int_{\|x\|=t}{ u(x) d \mathcal{H}^{n-1}} \lesssim_n \varepsilon + \mathcal{H}^{n-1}\left(  \left\{\|x\| = t: |u(x)| \geq \varepsilon \right\} \right) \cdot \|u\|_{L^{\infty}(\|x\| \leq 1)} \lesssim_n \varepsilon.$$
However, the third time is an integral over the expected exit time: starting in the center of the ball, that expected exit time ist $\gtrsim_n 1$ and thus
\begin{align*}
 \int_{\|x\| \leq 1/100}{  \mathbb{E} \int_{0}^{\tau}{ (\Delta u)(\omega_x(t)) dt}dx} &\gtrsim 
    \int_{\|x\| \leq 1/100}{  \mathbb{E} \int_{0}^{\tau}{1 dt}dx}\\
&\geq   \int_{\|x\| \leq 1/100}{  \mathbb{E} \tau dx} \gtrsim_n 1.
\end{align*}
This leads to a contradiction for $\varepsilon$ sufficiently small.
\end{proof}

\end{document}